\documentclass[a4paper,10pt]{amsart}
\usepackage[utf8]{inputenc}
\usepackage{amsxtra}
\usepackage{amsopn}
\usepackage{amsmath,amsthm,amssymb}
\usepackage{amscd}
\usepackage{amsfonts}
\usepackage{latexsym}
\usepackage{verbatim}
\usepackage{MnSymbol}
\usepackage{xcolor}
\usepackage{tikz-cd}
\usepackage{hyperref}
\usepackage[shortlabels]{enumitem}

\newcommand{\la}{\langle}
\newcommand{\ra}{\rangle}
\newcommand{\lv}{\lVert}
\newcommand{\rv}{\rVert}
\theoremstyle{plain}
\newtheorem{theorem}{Theorem}[section]
\newtheorem*{theorem*}{Theorem}

\newtheorem{lemma}[theorem]{Lemma}
\newtheorem{proposition}[theorem]{Proposition}
\newtheorem{corollary}[theorem]{Corollary}
\newtheorem{remark}[theorem]{Remark}

\newtheorem*{mt*}{Main Theorem}
\newcommand\C{{\mathbb C}}

\renewcommand\phi{{\varphi}}

\newcommand\N{{\mathbb N}}

\renewcommand\H{{\mathcal H}}

\newcommand{\del}{\partial}
\newcommand{\delbar}{{\overline{\del}}}

\newcommand{\cinf}{\mathcal{C}^\infty}

\let\inf\undefined
\DeclareMathOperator*{\inf}{inf\vphantom{p}}

\DeclareMathOperator{\im}{{im\,}}

\DeclareMathOperator{\D}{\mathcal{D}}
\DeclareMathOperator{\Gr}{Gr}

\DeclareMathOperator{\gr}{Gr}

\let\phi\varphi
\let\c\overline

\title[An $L^2$-$\del\delbar$-Lemma on a class of complete K\"ahler manifolds]{An $L^2$-$\del\delbar$-Lemma on a class of\\ complete K\"ahler manifolds}

\author{Riccardo Piovani}
\address{Dipartimento di Matematica e Fisica \lq\lq Niccolò Tartaglia\rq\rq,
Universit\`{a} Cattolica del Sacro Cuore, Via della Garzetta 48,
25133 Brescia, Italy}
\email{riccardo.piovani@unicatt.it}

\thanks{\newline 
The author is partially supported by GNSAGA of INdAM
}

\begin{document}

\begin{abstract}
    We prove an $L^2$-$\del\delbar$-Lemma involving smooth square integrable forms on complete K\"ahler manifolds, provided that the unique self-adjoint extension of the Hodge Laplacian on the Hilbert space of $L^2$-forms has a gap in its spectrum near zero. This generalises the classical $\del\delbar$-Lemma on compact K\"ahler manifolds. We also prove that on complete K\"ahler manifolds the elliptic Aeppli and Bott-Chern Laplacians are essentially self-adjoint.
\end{abstract}

\maketitle

\section{Introduction}
On compact K\"ahler manifolds, the $\del\delbar$-Lemma \cite[Lemmas 5.11, 5.15]{DGMS} states that for every smooth, complex-valued $k$-form $\alpha\in A^k_\C$ which is $\del$- and $\delbar$-closed, \textit{i.e.}, $\del\alpha=\delbar\alpha=0$, the following conditions are equivalent:
\begin{enumerate}
    \item $\alpha=\del\delbar\beta$, for $\beta\in A^{k-2}_\C$;
    \item $\alpha=\del\gamma$, for $\gamma\in  A^{k-1}_\C$;
    \item $\alpha=\delbar\zeta$, for $\zeta\in  A^{k-1}_\C$;
    \item $\alpha=\del\eta+\delbar\theta$, for $\eta,\theta\in  A^{k-1}_\C$;
    \item $\alpha=d\lambda$, for $\lambda\in  A^{k-1}_\C$,
\end{enumerate}
where the exterior derivative decomposes as $d=\del+\delbar$.
The validity of the above statement is invariant under holomorphic birational maps between compact complex manifolds \cite[Theorem 5.22]{DGMS}, and has important cohomological and topological implications \cite[Remark 5.15]{DGMS}: all the complex cohomology spaces (Aeppli, Bott-Chern, Dolbeault cohomology and its conjugate) are isomorphic, and the $k$-de Rham cohomology group is isomorphic to the direct sum of all the $(p,q)$-Dolbeault cohomology groups for $p+q=k$. A numerical characterisation of the validity of the $\del\delbar$-Lemma involving the dimensions of de Rham, Aeppli and Bott-Chern cohomology was proven in \cite{AT}.

The main purpose of this paper is to generalise the $\del\delbar$-Lemma to an $L^2$-$\del\delbar$-Lemma involving square integrable forms on a class of complete K\"ahler manifolds. We now introduce the result. Given a Hermitian manifold $(M,g)$, denote by $d^*,\del^*,\delbar^*$ the $L^2$-formal adjoints of the differential operators $d,\del,\delbar$, and by $\Delta_D:=DD^*+D^*D$ for $D\in\{d,\del,\delbar\}$ the Hodge, $\del$- and Dolbeault Laplacians, which are second-order formally self-adjoint elliptic operators. It is well known that if the metric is complete, then all these second-order Laplacians are essentially self-adjoint, namely, their restrictions to smooth compactly supported forms have a unique self-adjoint extension as unbounded operators between the Hilbert spaces of $L^2$-forms $L^2\Lambda^\bullet_\C:=\oplus_{k\in\N}L^2\Lambda^k_\C$, \textit{i.e.}, the space of measurable complex-valued forms with finite $L^2$-norm. These self-adjoint operators are non-negative, therefore their spectrum is contained in $[0,+\infty)$. We say that such an operator has a \emph{spectral gap} if its spectrum is contained in $\{0\}\cup[C,+\infty)$ for some $C>0$. This is equivalent to the image of the operator being closed (see Lemma \ref{lemma spectral gap equiv}). 

On the other hand, if the Hermitian metric $g$ is K\"ahler, \textit{i.e.}, its fundamental $2$-form is closed, then by K\"ahler identities the above second-order Laplacians coincide up to a factor, namely, $\Delta_d=2\Delta_\delbar=2\Delta_\del$. As a consequence, if the metric is both K\"ahler and complete, then all the self-adjoint extensions of these Laplacians coincide up to a factor. In the introduction we denote the unique self-adjoint extension of the Hodge Laplacian by $\Delta$. 

Besides compact K\"ahler manifolds, where the spectrum is discrete, examples of complete K\"ahler manifolds where $\Delta$ has a spectral gap on the whole algebra of forms $L^2\Lambda^\bullet_\C$ are given by complete K\"ahler $d$-bounded manifolds \cite[Theorem 1.4.A]{G}, where $d$-bounded means that the fundamental form $\omega$ is not only closed but also exact $\omega=d\eta$, and $\eta$ is bounded in the pointwise norm. Explicit examples of complete K\"ahler $d$-bounded manifolds are: complete simply connected K\"ahler manifolds with sectional curvature bounded from above by a negative constant \cite[0.1.B]{G} \cite[Lemma 3.2]{CY}, Hermitian symmetric spaces of non compact type \cite[0.1.C'']{G} \cite[Proposition 8.6]{Ba}, hyperconvex domains in Stein manifolds \cite[0.3.A(b)]{G} \cite[Proposition 2.2]{Do}, strictly pseudoconvex domains in $\C^n$ with the Bergman metric \cite[Proposition 3.4]{Do}, bounded homogeneous domains in $\C^n$ \cite[Theorem 1]{KO}.
We expect this $L^2$-$\del\delbar$-Lemma to be useful for further studies of the above classes of complete K\"ahler manifolds. For some recent results related to complete K\"ahler $d$-bounded manifolds we refer to \cite{PT1,BDET,H}. The notion of an $L^2$-$\del\delbar$-Lemma is also of interest on normal coverings of compact complex manifolds, in light of the inequalities between $L^2$ invariants proved in \cite{HP,BP}.

We use the notation $L^2A^k_\C$ to denote the space of smooth $L^2$-forms $L^2\Lambda^k_\C\cap A^k_\C$. The following theorem is the main result of the paper.
\begin{theorem}[{see Theorem \ref{theorem l2 del delbar lemma}}]\label{theorem intro l2 del delbar lemma}
    Let $(M,g)$ be a complete K\"ahler manifold such that the self-adjoint Laplacian $\Delta$ has a spectral gap in $L^2\Lambda^{k}_\C$. Given a smooth $L^2$-form $\alpha\in L^2A^k_\C$ which satisfies $\del\alpha=\delbar\alpha=0$, the following conditions are equivalent:
    \begin{enumerate}
        \item $\alpha=\del\delbar\beta$, for $\beta\in L^2 A^{k-2}_\C$;
        \item $\alpha=\del\gamma$, for $\gamma\in L^2 A^{k-1}_\C$;
        \item $\alpha=\delbar\zeta$, for $\zeta\in L^2 A^{k-1}_\C$;
        \item $\alpha=\del\eta+\delbar\theta$, for $\eta,\theta\in L^2 A^{k-1}_\C$;
        \item $\alpha=d\lambda$, for $\lambda\in L^2 A^{k-1}_\C$.
    \end{enumerate}
\end{theorem}

The proof of Theorem \ref{theorem intro l2 del delbar lemma} relies primarily on new spectral properties of some self-adjoint operators introduced in \cite{HP} as part of the development of the Aeppli and Bott-Chern versions of the $L^2$-Hodge theory of general Hermitian manifolds and of complete K\"ahler manifolds. More precisely, we prove that if $\Delta$ has a spectral gap, the unique self-adjoint extensions of the fourth-order elliptic Aeppli and Bott-Chern Laplacians have a spectral gap, thereby allowing one to use arguments based on elliptic regularity. Since this constitutes the core of the proof, we now discuss these Aeppli and Bott–Chern Laplacians in more detail.

Given a compact Hermitian manifold, there is a unique choice for the second-order Hodge, $\del$- and Dolbeault Laplacians $\Delta_D=DD^*+D^*D$, for $D\in\{d,\del,\delbar\}$, such that the corresponding kernels are isomorphic to de Rham, $\del$- and Dolbeault cohomology, via classical Hodge theory. In the Aeppli and Bott-Chern cases the situation is slightly different: there are multiple possible choices for the associated Laplacians. We recall the definitions of Aeppli and Bott-Chern cohomology via the related differential complex. We denote the spaces of $(p,q)$-forms by $A^{p,q}$, so that $A^k_\C=\oplus_{p+q=k}A^{p,q}$, and the Hilbert space of $L^2$-$(p,q)$-forms by $L^2\Lambda^{p,q}$. For any choice of integers $(p,q)$ we consider the complex
\[
\dots\longrightarrow A^{p-1,q-2}\oplus A^{p-2,q-1}\overset{\delbar\oplus\del}{\longrightarrow} A^{p-1,q-1}\overset{\del\delbar}{\longrightarrow} A^{p,q}\overset{\del+\delbar}{\longrightarrow} A^{p+1,q}\oplus A^{p,q+1}{\longrightarrow}\dots
\]
where $\delbar\oplus\del$ operates on $A^{p-1,q-2}\oplus A^{p-2,q-1}$ as $\delbar$ on $A^{p-1,q-2}$ plus $\del$ on $A^{p-2,q-1}$. The Aeppli and Bott-Chern cohomology spaces are defined as
\begin{align*}
H^{p-1,q-1}_A:=\frac{\ker\del\delbar\cap A^{p-1,q-1}}{\im\delbar\oplus\del},&&H^{p,q}_{BC}:=\frac{\ker(\del+\delbar)\cap A^{p,q}}{\im\del\delbar}.
\end{align*}
The \lq\lq natural" Aeppli and Bott-Chern Laplacians are then defined as
\begin{align*}
\Delta_A&:=\delbar^*\del^*\del\delbar+(\delbar\oplus\del)(\delbar\oplus\del)^*=\delbar^*\del^*\del\delbar+\del\del^*+\delbar\delbar^*,\\
\Delta_{BC}&:=\del\delbar\delbar^*\del^*+(\del+\delbar)^*(\del+\delbar)=\del\delbar\delbar^*\del^*+\del^*\del+\delbar^*\delbar.
\end{align*}
The kernels of these operators are isomorphic to the Aeppli and Bott-Chern cohomology spaces, but they are not elliptic \cite[Proposition 2.1]{S}. The first elliptic fourth-order operators whose kernels were shown to be isomorphic to the Aeppli and Bott-Chern cohomologies were defined in \cite{KS} as
\begin{align*}
\tilde\Delta_{A} &:=
\del\delbar\delbar^*\del^*+
\delbar^*\del^*\del\delbar+
\del\delbar^*\delbar\del^*+\delbar\del^*\del\delbar^*+
\del\del^*+\delbar\delbar^*,\\
\tilde\Delta_{BC} &:=
\del\delbar\delbar^*\del^*+
\delbar^*\del^*\del\delbar+\del^*\delbar\delbar^*\del+\delbar^*\del\del^*\delbar
+\del^*\del+\delbar^*\delbar.
\end{align*}
We will refer to these operators as the Kodaira-Spencer Laplacians. Another pair of elliptic fourth-order operators whose kernels are isomorphic to the Aeppli and Bott-Chern cohomologies were defined in \cite{V} as
\begin{align*}
\square_{A}&:=\delbar^*\del^*\del\delbar+((\delbar\oplus\del)(\delbar\oplus\del)^*)^2=\delbar^*\del^*\del\delbar+(\del\del^*+\delbar\delbar^*)^2,\\
\square_{BC}&:=\del\delbar\delbar^*\del^*+((\del+\delbar)^*(\del+\delbar))^2=\del\delbar\delbar^*\del^*+(\del^*\del+\delbar^*\delbar)^2.
\end{align*}
We will refer to these operators as the Varouchas Laplacians.

In \cite[Corollary 8.12]{HP} it was proven that on a complete K\"ahler manifold, if $\Delta$ has a spectral gap in $L^2\Lambda^\bullet_\C$, then there are self-adjoint extensions of $\Delta_A$ and $\Delta_{BC}$ having a spectral gap. Being $\Delta_A$ and $\Delta_{BC}$ non-elliptic, this was not sufficient to prove Theorem \ref{theorem intro l2 del delbar lemma}. In fact, to prove our main result, we use the spectral gap of the elliptic Varouchas Laplacians, which is shown in Theorem \ref{theorem spectral gap varouchas laplacians}. We also prove the spectral gap of the Kodaira-Spencer Laplacians in Theorem \ref{theorem spectral gap kodaira spencer laplacians}. We remark that the spectral gap assumption of $\Delta$ is essential for a statement which directly generalises the classical $\del\delbar$-Lemma such as Theorem \ref{theorem intro l2 del delbar lemma} does. 
Indeed, by Lemma \ref{lemma spectral gap equiv 2}, in the absence of a spectral gap of $\Delta$ in $L^2\Lambda^k_\C$, the image of the closed extensions of $d$ (resp. $\del$, $\delbar$) acting either on $(k-1)$-forms or on $k$-forms fails to be closed. 
For a weaker statement on complete K\"ahler manifolds without assuming a spectral gap, see Theorem \ref{theorem reduced l2 del delbar lemma} (cf. \cite[Corollary 8.6]{HP}).

We also present further developments in the Aeppli and Bott-Chern $L^2$-Hodge theory.
\begin{theorem}[{see Theorems \ref{theorem kodaira spencer self-adjoint}, \ref{theorem varouchas essentially self-adjoint} and Corollary \ref{corollary del delbar strong weak complete kahler}}]\label{theorem intro self adjoint}
Let $(M,g)$ be a complete K\"ahler manifold. Then $\tilde\Delta_{A}$, $\tilde\Delta_{BC}$, $\square_{A}$, $\square_{BC}$ restricted to smooth compactly supported forms have unique self-adjoint extensions. In particular the strong and weak closed extensions of $\del\delbar$ on $L^2\Lambda^\bullet_\C$ coincide.
\end{theorem}
The proofs are obtained by: viewing the Kodaira-Spencer Laplacians as small perturbations of the squared Dolbeault Laplacian $\Delta_\delbar^2$ and applying the Kato-Rellich Theorem; comparing the Varouchas Laplacians to the squared Dolbeault Laplacian $\Delta_\delbar^2$ via the K\"ahler identities and using a generalisation of a criterion for self-adjointness due to Kohn (see Theorem \ref{theorem sum kohn}).

The paper is organised as follows. In Section \ref{section preliminaries}, we present preliminaries on Hilbert complexes, self-adjoint operators, spectral gaps, minimal and maximal closed extensions of differential operators, differential complexes on complex manifolds, $L^2$ extensions of second and fourth-order complex Laplacians, $L^2$ Hodge theory on complete K\"ahler manifolds. The new results of this section are contained in Subsection \ref{subsection complete kahler}.   
In Section \ref{section spectral gap}, we prove the spectral gaps of the elliptic Aeppli and Bott-Chern Laplacians. In Theorem \ref{theorem closed image del delbar} we also present a new proof of the closedness of the image of the closed extensions of the operator $\del\delbar$, originally proved in \cite[Theorem 8.10]{HP}. Finally, in Section \ref{section l2 del delbar lemma}, we prove our main result Theorem \ref{theorem intro l2 del delbar lemma}. In Remark \ref{remark bounded geometry} we also point out an improvement of Theorem \ref{theorem intro l2 del delbar lemma} when the manifold is of bounded geometry.

\medskip\medskip
\noindent{\em Acknowledgments.} 
The author would like to thank the referees for their useful comments and remarks.

\section{Preliminaries}\label{section preliminaries}

\subsection{Hilbert complexes}\label{subsection hilbert complexes}

A \emph{Hilbert complex} consists of a complex of Hilbert spaces $\H_i$ along with linear operators $D_i:\H_i\to\H_{i+1}$ which are \emph{densely defined} ($D_i$ is defined on a domain $\D(D_i)$ which is a dense subspace of $\H_i$) and \emph{closed} (the graph of $D_i$ is closed in $\H_i\times\H_{i+1}$) of the form
\[
\H_0\overset{D_0}{\longrightarrow}\H_1\overset{D_1}\longrightarrow\dots\overset{D_{n-2}}\longrightarrow\H_{n-1}\overset{D_{n-1}}{\longrightarrow}\H_n
\]
such that $\im D_i\subseteq \ker D_{i+1}$. The notion of \emph{Hilbert complex} was systematically analysed by Br\"uning and Lesch in \cite{BL}.
To any Hilbert complex we can associate its \emph{cohomology} $H_i$ and \emph{reduced cohomology} $\bar H_i$
\[
H_i:=\frac{\ker D_i}{\im D_{i-1}},\ \ \ \bar H_i:=\frac{\ker D_i}{\c{\im D_{i-1}}}.
\]

We denote by $P^t:\H_2\to\H_1$ the Hilbert adjoint of any densely defined operator $P:\H_1\to\H_2$, which is defined by the relation
\begin{align*}
\la P x, y \ra_2=\la x, P^ty\ra_1 &&\forall x\in \D(P),\ \forall y\in \D(P^t),
\end{align*}
with domain
\[
\D(P^t):=\{y\in\H_2\,:\,\H_1\ni x\mapsto\la Px,y\ra_2\text{ is continuous}\},
\]
where $\la\cdot,\cdot\ra_i$ is the inner product on $\H_i$.
Recall that $P^t$ is always closed since its graph, with exchanged coordinates,  satisfies
\begin{equation}\label{eq graph perp}
\{(P^ty,y)\in\H_1\times\H_2\,|\,y\in\D(P^t)\}=\Gr(-P)^\perp.
\end{equation}
Moreover, if $P$ is also closed, then $\D(P^t)^\perp=\{0\}$ and so $P^t$ is also densely defined (see \cite[Theorem VIII.1]{RS1}). In this case
\begin{align*}
 P^{tt}=P,&&   \ker P^\perp=\c{\im P^t}.
\end{align*}
A fundamental property of Hilbert complexes is the following.

\begin{lemma}[{\cite[Lemma 2.1]{BL}}]
\label{lemma decomposition hilbert complexes}
Given a Hilbert complex $D_i:\H_i\to\H_{i+1}$, there are orthogonal decompositions
\begin{align*}
\H_i&=\ker D_i\cap\ker D_{i-1}^t\oplus\c{\im D_{i-1}}\oplus\c{\im D_i^t},\\
\ker D_i&=\ker D_i\cap\ker D_{i-1}^t\oplus\c{\im D_{i-1}},
\\ \ker D_{i-1}^t &= \ker D_i\cap\ker D_{i-1}^t \oplus \c{\im D_i^t}.
\end{align*}
\end{lemma}
The second decomposition in Lemma \ref{lemma decomposition hilbert complexes} provides an isomorphism
\[
\ker D_i\cap\ker D_{i-1}^t\simeq \bar H_i.
\]

\subsection{Self-adjoint operators}

A linear operator $P:\H\to\H$ on a Hilbert space $\H$ is said to be \emph{self-adjoint} if $P=P^t$, with equality of domains. The operator $P$ is said to be \emph{non-negative} if $\la Px, x\ra\ge0$ for all $x\in\D(P)$. We describe a couple of constructions of non-negative self-adjoint operators. 

\begin{theorem}[{\cite[Theorem X.25]{RS2}}
]\label{theorem composition with adjoint}
Let $P:\H_1\to\H_2$ be a closed and densely defined linear operator between Hilbert spaces $\H_1,\H_2$. Then the operator $P^tP$ defined by $(P^tP)x=P^t(Px)$ on the domain
\[
\D(P^t P):=\{x\in\D(P)\,|\,Px\in\D(P^t)\}
\]
is non-negative and self-adjoint.
\end{theorem}
In particular, if $P$ is self-adjoint, then we can define the non-negative and self-adjoint operator $P^2$.

The following theorem is a straightforward generalisation of \cite[Proposition 2.3]{K}. We present a self-contained proof for the convenience of the reader.

\begin{theorem}\label{theorem sum kohn}
Let $P_1,\dots,P_n:\H\to\H$ be non-negative and self-adjoint operators on a Hilbert space $\H$ such that
\[
\im P_i\subseteq\ker P_j\qquad\text{for all } i< j.
\]
Then the operator $P:=P_1+\dots+P_n$ defined by $Px=P_1x+\dots+P_nx$ on the domain
\[
\D(P):=\D(P_1)\cap\dots\cap\D(P_n)
\]
is non-negative and self-adjoint. Moreover, $(I+P)^{-1}$ is bounded, defined on the whole $\H$, and given by
\begin{equation}\label{equation resolvent sum kohn}
(I+P)^{-1}=\sum_{i=1}^n(I+P_i)^{-1}-(n-1)I.
\end{equation}
\end{theorem}
\begin{proof}
Since $P_i$ is non-negative and self-adjoint, $-1$ does not belong to its spectrum \cite[Theorem X.1]{RS2}, \textit{i.e.}, $R_i:=(I+P_i)^{-1}$ is a bounded operator defined on the whole $\H$, with $\im R_i=\D(P_i)$ (see Subsection \ref{subsection spectral gap} for the definition of spectrum). The operator $R_i$ is also self-adjoint: for all $h,k\in\H=\D(R_i)=\D(R_i^t)$
\[
\la R_ih,k\ra=\la R_ih,(I+P_i)R_ik\ra=\la (I+P_i)R_ih,R_ik\ra=\la h,R_ik\ra.
\]

We now show that the bounded and self-adjoint operator
\[
S:=\sum_{i=1}^nR_i-(n-1)I
\]
satisfies $(I+P)S=I$. First note that the assumption $\im P_i\subseteq \ker P_j=(\im P_j)^\perp$ is equivalent to $\im P_i\perp\im P_j$. In particular the condition is symmetric and it holds for all $i\ne j$.
Now observe that
\[
R_i-I=R_i-(I+P_i)R_i=-P_iR_i,
\]
therefore $\im(R_i-I)\subseteq\im P_i\subseteq\ker P_j$ for all $j\neq i$. For every $k\in\{1,\dots,n\}$ we can write
\[
S=R_k+\sum_{i\neq k}(R_i-I)=R_k-\sum_{i\neq k}P_iR_i.
\]
Given $h\in\H$, we have $R_kh\in\D(P_k)$ and $P_iR_ih\in\ker P_k\subseteq\D(P_k)$ for $i\neq k$, hence $Sh\in\D(P_k)$ and
\[
P_kSh=P_kR_kh=h-R_kh.
\]
Since this holds for every $k$, it follows that $Sh\in\D(P)$ and
\[
(I+P)Sh=Sh+\sum_{k=1}^n(h-R_kh)=\sum_{i=1}^nR_ih-(n-1)h+nh-\sum_{k=1}^nR_kh=h.
\]
This proves $(I+P)S=I$; in particular $S$ is injective.

We now prove that the operator $T:=S^{-1}$, defined on the domain $\D(T):=\im S$, is self-adjoint. Since $S$ is self-adjoint and injective, we have $\c{\im S}=(\ker S)^\perp=\H$, therefore $T$ is densely defined.
Since $S=S^t$, for all $(u,v)\in\H\times\H$ and $h\in\H$ we have $\la (Sh,-h),(u,v)\ra=\la h,Su-v\ra$. Using this and \eqref{eq graph perp}, we get
\begin{align*}
\{(T^ty,y)\,|\,y\in\D(T^t)\}&=\Gr(-T)^\perp=\{(Sh,-h)\,|\,h\in\H\}^\perp\\
&=\{(u,Su)\,|\,u\in\H\}=\{(Tx,x)\,|\,x\in\D(T)\},
\end{align*}
that is $\Gr(T^t)=\Gr(T)$, \textit{i.e.}, $T^t=T$.

Finally, the equality $(I+P)S=I$ gives $\Gr(T)\subseteq \Gr(I+P)$. Since $T$ is surjective and $I+P$ is injective, being $\la(I+P)x,x\ra\ge\lv x\rv^2$ for all $x\in\D(P)$, we get $T=I+P$. Therefore $I+P$ is self-adjoint, which implies that $P$ is self-adjoint, and $(I+P)^{-1}=S$, which is \eqref{equation resolvent sum kohn}. The non-negativity of $P$ is clear.
\end{proof}

\begin{remark}\label{remark laplacian hilbert complex}
Given a Hilbert complex $D_i:\H_i\to\H_{i+1}$, then by Theorem \ref{theorem composition with adjoint} the operators $D_i^tD_i$ and $D_{i-1}D_{i-1}^t$ are non-negative and self-adjoint on $\H_i$. Moreover, since $\im D_{i-1}D_{i-1}^t\subseteq\im D_{i-1}\subseteq\ker D_i\subseteq\ker D_i^tD_i$
, then by Theorem \ref{theorem sum kohn} the Laplacian operator $\Delta_i:=D_i^tD_i+D_{i-1}D_{i-1}^t$ is non-negative and self-adjoint on $\H_i$. We finally observe that
\[
\ker \Delta_i=\ker D_i\cap \ker D_{i-1}^t.
\]
\end{remark}

\subsection{Closed image and spectral gap}\label{subsection spectral gap}

Given a closed and densely defined operator, the property that its image is closed has several well-known characterisations. We refer, \textit{e.g.}, to \cite[Lemma 4.3]{HP} for a proof.

\begin{lemma}\label{lemma im closed}
Let $P:\H_1\to\H_2$ be a closed and densely defined linear operator between Hilbert spaces. The following conditions are equivalent:
\begin{enumerate}[label=\upshape{\alph*)}]
\item $\im P$ is closed;
\item $\exists C>0 \text{ s.t. } \lv x\rv_1\le C\lv Px\rv_2$ for all $x\in\D(P)\cap \c{\im P^t}$;
\item $\im P^t$ is closed;
\item $\exists C>0 \text{ s.t. } \lv y\rv_2\le C\lv P^ty\rv_1$ for all $y\in\D(P^t)\cap \c{\im P}$.
\end{enumerate}
\end{lemma}

We say that $\lambda\in \mathbb{C}$ belongs to the \emph{spectrum} $\sigma(P)$ of an unbounded linear operator $P:\H\to\H$ if there exists no bounded linear operator $B:\H\to\H$ such that 
\begin{itemize}
    \item[1)] $B(P-\lambda I)x = x $ for all $x \in \mathcal{D}(P)$,
    \item[2)] $Bx\in \mathcal{D}(P)$ and $(P-\lambda I)Bx = x$ for all $x \in \mathcal{H}$;
\end{itemize}
in other words, if $P-\lambda I$ has no bounded inverse. By \cite[Theorem X.1]{RS2}, the spectrum of a self-adjoint operator is a subset of the real axis, and the spectrum of a non-negative self-adjoint operator is a subset of the non-negative real axis. Furthermore, we say that a non-negative self-adjoint operator $P$ has a \emph{spectral gap} if it has the property $\inf(\sigma (P) \backslash \{0\})=C>0$, or equivalently if $\sigma(P)\subseteq\{0\}\cup[C,+\infty)$ with $C>0$. 

We recall the following characterisation of the spectral gap condition. We include a proof for convenience of the reader.

\begin{lemma}\label{lemma spectral gap equiv}
Let $P$ be a non-negative self-adjoint operator on a Hilbert space $\H$. Then the following conditions are equivalent:
\begin{enumerate}[label=\upshape{\alph*)}]
\item $P$ has a spectral gap;
\item $\exists C>0 \text{ s.t. } \la Px,x\ra\ge C\la x,x\ra$ for all $x\in\D(P)\cap(\ker P)^\perp$;
\item $\im P$ is closed;
\item $P^2$ has a spectral gap.
\end{enumerate}
\end{lemma}
\begin{proof}
We need the Spectral Theorem for unbounded self-adjoint operators, for which we borrow the notation from \cite[Section A.5.4]{Gri}. Denoting by $\{E_\lambda\}$ the spectral resolution of $P$, for $x,y\in\D(P)$ we compute
\[
\la x,y \ra=\int_{\sigma(P)}d(E_\lambda x,y),
\]
\[
\la Px,y \ra=\int_{\sigma(P)}\lambda d(E_\lambda x,y),
\]
\[
\la Px,Py \ra=\la P^2x,y \ra=\int_{\sigma(P)}\lambda^2 d(E_\lambda x,y).
\]
Notice that $y\in\ker P$ iff
\[
\int_{\sigma(P)\cap(0,+\infty)}d(E_\lambda x,y)=0
\]
for all $x\in\D(P)$, for which
\[
\la x,y\ra=\int_{\sigma(P)\cap\{0\}}d(E_\lambda x,y).
\]
In particular, $x\in(\ker P)^\perp$ iff 
\[
\la x,y\ra=\int_{\sigma(P)\setminus\{0\}}d(E_\lambda x,y)
\]
for all $y\in\D(P)$. Having this established, the claimed equivalences become easy exercises. To show d)$\implies$a) we also make use of the Spectral Mapping Theorem of \cite[Section A.5.4]{Gri}.
\end{proof}

Finally, we recall another characterisation of the spectral gap condition for the Laplacian of a Hilbert complex.

\begin{lemma}[{\cite[Lemma 4.9]{HP}}]\label{lemma spectral gap equiv 2}
Given a Hilbert complex $D_i:\H_i\to\H_{i+1}$ with Laplacian $\Delta_i$, the following conditions are equivalent:
\begin{enumerate}[label=\upshape{\alph*)}]
\item $\Delta_i$ has a spectral gap;
\item $\exists C_i>0 \text{ s.t. } C_i\lv x\rv^2_i \le \lv D_{i-1}^tx\rv^2_{i-1}+\lv D_ix\rv^2_{i+1}$ for all $x\in\D(D_{i-1}^t)\cap\D(D_i)\cap(\ker \Delta_i)^\perp$;
\item $\im D_{i-1}$ and $\im D_i$ are closed.
\end{enumerate}
\end{lemma}

\subsection{Strong and weak extensions}

An \emph{extension} of a linear operator between Hilbert spaces $P:\H_1\to\H_2$ is a linear operator $P':\H_1\to\H_2$ such that $\D(P)\subseteq\D(P')$ and $Px=P'x$ for every $x\in\D(P)$, \textit{i.e.}, $\gr(P)\subseteq\gr(P')$. We will use the notation $P\subseteq P'$ to denote such an extension.

Let $(M,g)$ be a Riemannian manifold, $(E_1,h_1)$ and $(E_2,h_2)$ be Hermitian vector bundles on $M$, and let ${P}:\cinf(M,E_1)\to \cinf(M,E_2)$ be a differential operator.  We denote by ${P^*}:\cinf(M,E_2)\to \cinf(M,E_1)$ the \emph{formal adjoint} of $P$.

Denoting by $\cinf_0(M,E_j)$ the space of smooth sections of $E_j$ with compact support, we will indicate the restriction of $P$ to compactly supported sections as $P_0:\cinf_0(M,E_1) \to \cinf_0(M,E_2)$. This can be viewed as an unbounded and densely defined linear operator $P_0:L^2(M,E_1)\to L^2(M,E_2)$, with domain $\cinf_0(M,E_1)$. We recall the constructions of two canonical closed extensions of $P_0$.

The \emph{strong extension} $P_{s}$ (also called the \emph{minimal closed extension} $P_{min}$) is defined by taking the closure of the graph of ${P_0}$. Its domain $\D(P_s)$ is defined as
\begin{equation*}
\{u\in L^2(M,E_1)\,|\,\exists\{u_j\}_{j\in\N}\subset \cinf_0(M,E_1),\ \exists v\in L^2(M,E_2) \text{ s.t.}\ u_j\to u,\ Pu_j\to v\},
\end{equation*}
and $P_s u:=v$ for $u\in\D(P_s)$. By definition, the strong extension is the smallest closed extension of ${P_0}$.

The \emph{weak extension} $P_w$ (also called the \emph{maximal closed extension} $P_{max}$) is defined as the largest extension of ${P_0}$ which acts distributionally. Its domain $\D(P_w)$ is defined as
\begin{equation*}
\{u\in L^2(M,E_1)\,|\,\exists v\in L^2(M,E_2),\ \text{s.t.} \ \la v,w\ra_2=\la u,P^*w\ra_1, \ \forall w\in \cinf_0(M,E_2)\},
\end{equation*}
and $P_w u:=v$ for $u\in\D(P_w)$. Here $\la\cdot, \cdot \ra_i$ denotes the inner product defined on $L^2(M,E_i)$. Note that this is equivalent to saying $P_w$ is the Hilbert adjoint of $P^*$ restricted to smooth forms with compact support, \textit{i.e.}, $((P^*)_0)^t=P_w$. Moreover, the weak extension is an extension of the strong extension, that is $P_s\subseteq P_w$, and every closed extension $P'$ of $P_0$ which acts distributionally (namely, such that $(P')^t$ is an extension of $(P^*)_0$) is contained between the minimal and the maximal closed extensions. 

\begin{remark}\label{remark strong-weak}
Since a closable densely defined operator and its closure have the same adjoint by \cite[Theorem VIII.1]{RS1}, it follows that $((P^*)_s)^t=P_w$. This implies
\begin{equation*}\label{eq strong-weak}
(P^*)_s=(P_w)^t,\ \ \ (P^*)_w=(P_s)^t.
\end{equation*}
\end{remark}

A linear operator $L:\H\to\H$ on a Hilbert space $\H$ is  called \emph{essentially self-adjoint} if it has a unique self-adjoint extension. Equivalently, $L$ is essentially self-adjoint iff the closure of its graph defines a self-adjoint operator \cite[p. 256]{RS1}.

A differential operator $P:\cinf(M,E)\to\cinf(M,E)$ between sections of a Hermitian vector bundle $(E,h)$ is called \emph{formally self-adjoint} if $P=P^*$.

\begin{remark}
If a differential operator $P$ is formally self-adjoint, then, by Remark \ref{remark strong-weak}, $P_0$ is essentially self-adjoint if and only if $P_s=P_w$.
\end{remark}

We provide a number of technical well-known properties of the strong and weak extensions which will be needed in subsequent sections. To do so, we introduce the order relation $s\le w$ between strong and weak extensions, as well as the notation $a'=w$ if $a=s$ and $a'=s$ if $a=w$.

\begin{proposition}[{\cite[Section 5]{HP}}]\label{proposition properties strong weak}
Let ${P}:\cinf(M,E_1)\to\cinf(M,E_2)$ and ${Q}:\cinf(M,E_2)\to\cinf(M,E_3)$ be differential operators between smooth sections of Hermitian vector bundles. Then
\begin{enumerate}
    \item $\c{\im P_0}=\c{\im P_s}$;
    \item $\c{\im (QP)_s}\subseteq \c{\im Q_s}$;
    \item $\ker P_w\subseteq \ker(QP)_w$;
    \item if $QP= 0$, then
    $\c{\im P_a}\subseteq\ker Q_b$ for any $a,b\in\{s,w\}$ with $a\le b$;
    \item if $P_s\subseteq P'\subseteq P_w$, we have ${P} = P'$ when acting on $\cinf(M,E_1)\cap\D(P')$;
    \item there are equalities
$$\D(P_w)\cap\cinf(M,E_1)=\{\alpha\in L^2(M,E_1)\cap \cinf(M,E_1)\,|\,{P}\alpha\in L^2(M,E_2)\},$$ 
$$\ker {P}_w\cap \cinf(M,E_1)=\{\alpha\in L^2(M,E_1)\cap \cinf(M,E_1)\,|\,{P}\alpha=0\}.$$
\end{enumerate}
\end{proposition}

\subsection{From a differential complex to a Hilbert complex}
\label{subsection from differential to hilbert complex}
Given a Riemannian manifold $(M,g)$, let us consider some Hermitian vector bundles $(E_i,h_i)$ for $i=1,2,3$ and a complex of differential operators
\[
\cinf(M,E_1)\overset{P}\longrightarrow\cinf(M,E_2)\overset{Q}\longrightarrow\cinf(M,E_3)
\]
\textit{i.e.}, $Q\circ P=0$. It follows from \emph{(4)} of Proposition \ref{proposition properties strong weak} that
\[
L^2(M,E_1)\overset{P_a}\longrightarrow L^2(M,E_2)\overset{Q_b}\longrightarrow L^2(M,E_3)
\]
is a Hilbert complex for any choice of strong and weak extensions $a,b\in\{s,w\}$ such that $a\le b$.

If the above differential complex is defined by more than two operators, say $D_i:\cinf(M,E_i)\to\cinf(M,E_{i+1})$, then a Hilbert complex is defined by extensions $(D_i)_{a_i}:L^2(M,E_i)\to L^2(M,E_{i+1})$ for $a_i\in\{s,w\}$ such that $a_i\le a_{i+1}$ for all $i$.

\subsection{Differential complexes on complex manifolds}\label{subsection complexes}
Given a complex manifold $M$, we recall the following natural complexes defined on complex valued differential forms by the relations
\[
d^2=\del^2=\delbar^2=\del\delbar+\delbar\del=0.
\]
For any choice of integers $k,p,q$ we consider
\[
\dots\longrightarrow A^{k-1}_\C\overset{d}{\longrightarrow} A^{k}_\C\overset{d}{\longrightarrow} A^{k+1}_\C{\longrightarrow}\dots
\]
\[
\dots\longrightarrow A^{p-1,q}\overset{\del}{\longrightarrow} A^{p,q}\overset{\del}{\longrightarrow} A^{p+1,q}{\longrightarrow}\dots
\]
\[
\dots\longrightarrow A^{p,q-1}\overset{\delbar}{\longrightarrow} A^{p,q}\overset{\delbar}{\longrightarrow} A^{p,q+1}{\longrightarrow}\dots
\]
\[
\dots\longrightarrow A^{p-1,q-2}\oplus A^{p-2,q-1}\overset{\delbar\oplus\del}{\longrightarrow} A^{p-1,q-1}\overset{\del\delbar}{\longrightarrow} A^{p,q}\overset{\del+\delbar}{\longrightarrow} A^{p+1,q}\oplus A^{p,q+1}{\longrightarrow}\dots
\]
where $\delbar\oplus\del$ operates on $A^{p-1,q-2}\oplus A^{p-2,q-1}$ as $\delbar$ on $A^{p-1,q-2}$ plus $\del$ on $A^{p-2,q-1}$. By Subsection \ref{subsection from differential to hilbert complex}, once we fix a Hermitian metric on $M$, which induces natural Hermitian metrics on the vector bundles of forms, each of the previous complexes defines a Hilbert complex. We apply Lemma \ref{lemma decomposition hilbert complexes} to these Hilbert complexes and, taking into account of Remark \ref{remark strong-weak}, deduce the following general result.

\begin{theorem}[{\cite[Section 6 and Theorem 7.4]{HP}}]\label{theorem orthogonal decompositions}
Given a Hermitian manifold $(M,g)$, for $a,b\in\{s,w\}$ with $a\le b$, there are $L^2$ orthogonal decompositions
\begin{align*}
L^2\Lambda^k_\C&=\ker d_b\cap\ker d_{a'}^*\overset{\perp}{\oplus}\c{\im d_a}\overset{\perp}{\oplus}\c{\im d_{b'}^*},\\
L^2\Lambda^{p,q}&=\ker \del_b\cap\ker \del_{a'}^*\overset{\perp}{\oplus}\c{\im \del_a}\overset{\perp}{\oplus}\c{\im \del_{b'}^*},\\
L^2\Lambda^{p,q}&=\ker \delbar_b\cap\ker \delbar_{a'}^*\overset{\perp}{\oplus}\c{\im \delbar_a}\overset{\perp}{\oplus}\c{\im \delbar_{b'}^*},\\
L^2\Lambda^{p,q}&=\ker \del\delbar_b\cap\ker (\delbar\oplus\del)_{a'}^*\overset{\perp}{\oplus}\c{\im (\delbar\oplus\del)_a}\overset{\perp}{\oplus}\c{\im \delbar^*\del^*_{b'}},\\
L^2\Lambda^{p,q}&=\ker (\del+\delbar)_b\cap\ker \delbar^*\del^*_{a'}\overset{\perp}{\oplus}\c{\im \del\delbar_a}\overset{\perp}{\oplus}\c{\im (\del+\delbar)_{b'}^*}
\end{align*}
and
\begin{align*}
L^2\Lambda^k_\C\cap\ker d_b&=\ker d_b\cap\ker d_{a'}^*\overset{\perp}{\oplus}\c{\im d_a},\\
L^2\Lambda^{p,q}\cap\ker \del_b&=\ker \del_b\cap\ker \del_{a'}^*\overset{\perp}{\oplus}\c{\im \del_a},\\
L^2\Lambda^{p,q}\cap\ker \delbar_b&=\ker \delbar_b\cap\ker \delbar_{a'}^*\overset{\perp}{\oplus}\c{\im \delbar_a},\\
L^2\Lambda^{p,q}\cap\ker \del\delbar_b&=\ker \del\delbar_b\cap\ker (\delbar\oplus\del)_{a'}^*\overset{\perp}{\oplus}\c{\im (\delbar\oplus\del)_a},\\
L^2\Lambda^{p,q}\cap\ker (\del+\delbar)_b&=\ker (\del+\delbar)_b\cap\ker \delbar^*\del^*_{a'}\overset{\perp}{\oplus}\c{\im \del\delbar_a}.
\end{align*}
\end{theorem}

We end this subsection by recalling from \cite[Lemma 7.10]{HP} that on $L^2\Lambda^{p,q}$ it holds
\begin{align}
\ker(\del+\delbar)_w&=\ker\del_w\cap\ker\delbar_w,&\ker(\del^*+\delbar^*)_w&=\ker\del^*_w\cap\ker\delbar^*_w.\label{equation kernel}
\end{align}
Let us consider the formal adjoint operators $\del^*\oplus\delbar^*=(\del+\delbar)^*:A^{p+1,q}\oplus A^{p,q+1}\to A^{p,q}$ and $(\delbar\oplus\del)^*=\del^*+\delbar^*:A^{p,q}\to A^{p-1,q}\oplus A^{p,q-1}$.
It follows that
\begin{align*}
\im(\delbar\oplus\del)_s^\perp=\ker(\del^*+\delbar^*)_w=\ker\del^*_w\cap\ker\delbar^*_w=(\c{\im\del_s}+\c{\im\delbar_s})^\perp,
\end{align*}
and similarly
\begin{align*}
\im(\del^*\oplus\delbar^*)_s^\perp=(\c{\im\del^*_s}+\c{\im\delbar^*_s})^\perp.
\end{align*}

\subsection{Extensions of second- and fourth-order complex Laplacians}\label{subsection self adjoint laplacians}

Given a Hermitian manifold $(M,g)$, we recall the constructions of the extensions of the complex Laplacians from \cite[Sections 6,7,8]{HP}. 

The self-adjoint extensions of the Dolbeault Laplacian are simply given by the natural Laplacians associated with the Hilbert complex (see Remarks \ref{remark laplacian hilbert complex} and \ref{remark strong-weak}) arising from the Dolbeault complex: for $a,b\in\{s,w\}$ with $a\le b$
\[
\Delta_{\delbar,ab}:=\delbar_a\delbar^*_{a'}+\delbar^*_{b'}\delbar_b:L^2\Lambda^{p,q}\to L^2\Lambda^{p,q}.
\]
Similarly, we can build self-adjoint extensions
\begin{align*}
\Delta_{\del,ab}&:=\del_a\del^*_{a'}+\del^*_{b'}\del_b:L^2\Lambda^{p,q}\to L^2\Lambda^{p,q},\\
\Delta_{d,ab}&:=d_ad^*_{a'}+d^*_{b'}d_b:L^2\Lambda^{k}_\C\to L^2\Lambda^{k}_\C,\\
\Delta_{A,ab}&:=(\delbar\oplus\del)_a(\delbar\oplus\del)^*_{a'}+\delbar^*\del^*_{b'}\del\delbar_b:L^2\Lambda^{p,q}\to L^2\Lambda^{p,q},\\
\Delta_{BC,ab}&:=\del\delbar_a\delbar^*\del^*_{a'}+(\del+\delbar)^*_{b'}(\del+\delbar)_b:L^2\Lambda^{p,q}\to L^2\Lambda^{p,q},
\end{align*}
of the $\del$- and Hodge-de Rham Laplacians, and of the non-elliptic natural Aeppli and Bott-Chern Laplacians.

The extensions of the elliptic Aeppli and Bott-Chern Laplacians are defined as follows. We begin defining extensions of the Varouchas Laplacians \cite{V}. 
For $a,b,c\in\{s,w\}$ with $a\le b\le c$, we define on $L^2\Lambda^{p,q}$ the operators
\begin{align*}
\square_{A,ab}&:=((\delbar\oplus\del)_a(\delbar\oplus\del)^*_{a'})^2+\delbar^*\del^*_{b'}\del\delbar_b,\\
\square_{BC,bc}&:=\del\delbar_b\delbar^*\del^*_{b'}+((\del+\delbar)^*_{c'}(\del+\delbar)_c)^2,
\end{align*}
which are non-negative and self-adjoint with the domains given by repeatedly applying Theorems \ref{theorem composition with adjoint} and \ref{theorem sum kohn} as in Remark \ref{remark laplacian hilbert complex}. They are self-adjoint extensions of the elliptic Aeppli and Bott-Chern Laplacians $(\square_{A})_0$, $(\square_{BC})_0$.

We then define extensions of the Kodaira-Spencer Laplacians \cite{KS}. 
For $a,b,c\in\{s,w\}$ with $a\le b\le c$, we define on $L^2\Lambda^{p,q}$ the operators
\begin{align*}
\tilde\Delta_{A,ab}:=& \delbar^*\del^*_{b'}\del\delbar_{b}+
\del\delbar_s\delbar^*\del^*_w+\delbar\del^*_s\del\delbar^*_w+ \del\delbar^*_s\delbar\del^*_w +
(\delbar \oplus \del)_a(\delbar \oplus \del)^*_{a'},\\
\tilde\Delta_{BC,bc}:=&\del\delbar_{b}\delbar^*\del^*_{b'}+
\delbar^*\del^*_s\del\delbar_w+
\del^*\delbar_s\delbar^*\del_w+\delbar^*\del_s\del^*\delbar_w+
(\del+\delbar)^*_{c'}(\del+\delbar)_c,
\end{align*}
which are non-negative with the domains given by applying Theorem \ref{theorem composition with adjoint} and then taking the intersections of the domains. They are extensions of the elliptic Aeppli and Bott-Chern Laplacians $(\tilde\Delta_{A})_0$, $(\tilde\Delta_{BC})_0$.

A simple computation and the use of \emph{(3)} of Proposition \ref{proposition properties strong weak} allow us to describe the kernels of the above operators:
\begin{align*}
&\ker\Delta_{A,ab}=\ker\square_{A,ab}=\ker\tilde\Delta_{A,ab}= \ker\del\delbar_{b}\cap\ker(\del^*+\delbar^*)_{a'},\\
&\ker\Delta_{BC,bc}=\ker\square_{BC,bc}=\ker\tilde\Delta_{BC,bc}=\ker\delbar^*\del^*_{b'}\cap\ker(\del+\delbar)_c.
\end{align*}

We also need to define extensions of just the fourth-order part of the elliptic Aeppli and Bott-Chern Laplacians $(\tilde\Delta_{A})_0$, $(\tilde\Delta_{BC})_0$. For $b\in\{s,w\}$, we define
\begin{align*}
\tilde\Delta_{A,b,4}:=& \delbar^*\del^*_{b'}\del\delbar_{b}+
\del\delbar_s\delbar^*\del^*_w+\delbar\del^*_s\del\delbar^*_w+ \del\delbar^*_s\delbar\del^*_w,\\
\tilde\Delta_{BC,b,4}:=&\del\delbar_{b}\delbar^*\del^*_{b'}+
\delbar^*\del^*_s\del\delbar_w+
\del^*\delbar_s\delbar^*\del_w+\delbar^*\del_s\del^*\delbar_w,
\end{align*}
which are non-negative with the domains given by applying Theorem \ref{theorem composition with adjoint} and then taking the intersections of the domains.

\subsection{K\"ahler manifolds}\label{subsection kahler}

Given a K\"ahler manifold $(M,g)$, thanks to the K\"ahler identities 
\begin{align}\label{eq kahler identities}
\del\delbar^*+\delbar^*\del=0, \ \ \ \ \ \  \delbar\del^*+\del^*\delbar=0,
\end{align}
we have
\[
\Delta_\delbar=\Delta_\del=\frac12\Delta_d
\]
and
\begin{align*}
\tilde\Delta_A&=\Delta_\delbar\Delta_\delbar+\del\del^*+\delbar\delbar^*,\\
\tilde\Delta_{BC}&=\Delta_\delbar\Delta_\delbar+\del^*\del+\delbar^*\delbar.
\end{align*}
Moreover, using \emph{(4)} of Proposition \ref{proposition properties strong weak} and \eqref{eq kahler identities}, we can apply Theorem \ref{theorem sum kohn} to $\tilde\Delta_{A,b,4}$, $\tilde\Delta_{BC,b,4}$, which are therefore self-adjoint. For each operator, there are six inclusions to check. We prove two inclusions as examples. To check 
\[
\im \del\delbar^*_s\delbar\del^*_w\subseteq\ker \delbar\del^*_s\del\delbar^*_w,
\]
we argue that $\im\del\delbar^*_s\subseteq \ker\del\delbar^*_w$ since $\del\delbar^*\del\delbar^*=-\del\delbar^*\delbar^*\del=0$ by \eqref{eq kahler identities} and then by \emph{(4)} of Proposition \ref{proposition properties strong weak}. Similarly, for $\im \del\delbar_s\delbar^*\del^*_w\subseteq\ker \del\delbar^*_s\delbar\del^*_w$, we use $\delbar\del^*\del\delbar=-\del^*\delbar\del\delbar=\del^*\del\delbar\delbar=0$.

\subsection{Complete K\"ahler manifolds}\label{subsection complete kahler}

Given a complete Hermitian manifold $(M,g)$, it is well known that $\delta_s=\delta_w$ for the first-order operators $\delta=d,\del,\delbar,\delbar\oplus\del$ and their formal adjoints. We refer to \cite[Lemma 4]{AV} for the proof when $\delta=\delbar$; the other cases are similar. In the rest of the paper, under the assumption of completeness, we will write either $\delta_s$ or $\delta_w$ depending on which is more convenient at the time.
Moreover,  all positive integer powers of $(d+d^*)_0,(\del+\del^*)_0,(\delbar+\delbar^*)_0$ as operators $A^{\bullet,\bullet}_0\to A^{\bullet,\bullet}_0$ are essentially self-adjoint \cite[Section 3.B]{Che}. For example, $\Delta_{\delbar,sw}$ is the unique self-adjoint extension of $(\Delta_\delbar)_0=(\delbar+\delbar^*)_0^2$ and $(\Delta_{\delbar,sw})^2$ (which can be defined via Theorem \ref{theorem composition with adjoint}) is the unique self-adjoint extension of $(\Delta_\delbar^2)_0=(\delbar+\delbar^*)_0^4$.

Therefore, given a complete K\"ahler manifold $(M,g)$, it follows by Subsection \ref{subsection kahler} that
\begin{equation}\label{eq second order laplacian complete kahler}
\Delta_{\delbar,sw}=\Delta_{\del,sw}=\frac12\Delta_{d,sw}
\end{equation}
and furthermore for any $b\in\{s,w\}$
\begin{align}\label{equation square dolbeault aeppli bc 4}
\tilde\Delta_{A,b,4}=(\Delta_{\delbar,sw})^2=\tilde\Delta_{BC,b,4}
\end{align}
by the essential self-adjointness of $(\Delta_\delbar^2)_0$. Thanks to this, we can see $\tilde\Delta_{A,ab}$ and $\tilde\Delta_{BC,bc}$ as perturbations of $(\Delta_{\delbar,sw})^2$. Measuring the smallness of these perturbations, we are able to prove the self-adjointness of the Kodaira-Spencer Laplacians $\tilde\Delta_{A,ab}$ and $\tilde\Delta_{BC,bc}$ plus the essential self-adjointness of  $(\tilde\Delta_{A})_0$ and $(\tilde\Delta_{BC})_0$. This is obtained via the Kato-Rellich Theorem as follows.

\begin{theorem}\label{theorem kodaira spencer self-adjoint}
Let $(M,g)$ be a complete K\"ahler manifold and let $a,b,c\in\{s,w\}$ with $a\le b\le c$. Then the operators $\tilde\Delta_{A,ab}$ and $\tilde\Delta_{BC,bc}$ are non-negative and self-adjoint, with
\[
\D(\tilde\Delta_{A,ab})=\D(\tilde\Delta_{BC,bc})=\D((\Delta_{\delbar,sw})^2).
\]
Moreover, $\tilde\Delta_{A,ab}$ and $\tilde\Delta_{BC,bc}$ are the unique self-adjoint extensions of $(\tilde\Delta_A)_0$ and $(\tilde\Delta_{BC})_0$, respectively. In particular, they do not depend on the choice of $a,b,c$.
\end{theorem}
\begin{proof}
We prove the statement for $\tilde\Delta_{A,ab}$. Set
\[
Q:=(\delbar\oplus\del)_a(\delbar\oplus\del)^*_{a'},
\]
which is non-negative and self-adjoint by Theorem \ref{theorem composition with adjoint}. By definition of $\tilde\Delta_{A,ab}$ and $\tilde\Delta_{A,b,4}$, and by \eqref{equation square dolbeault aeppli bc 4}, we have
\begin{gather*}
\D(\tilde\Delta_{A,ab})=\D(\tilde\Delta_{A,b,4})\cap\D(Q)=\D((\Delta_{\delbar,sw})^2)\cap\D(Q),\\
\tilde\Delta_{A,ab}=(\Delta_{\delbar,sw})^2+Q.
\end{gather*}

We first show that $\D(\Delta_{\delbar,sw})\subseteq\D(Q)$ and
\begin{equation}\label{equation estimate second order part}
\lv Qx\rv\le2\lv\Delta_{\delbar,sw} x\rv\qquad\forall x\in\D(\Delta_{\delbar,sw}).
\end{equation}
Let $u\in A^{\bullet,\bullet}_0$. By \emph{(5)} of Proposition \ref{proposition properties strong weak}, $Qu=\del\del^*u+\delbar\delbar^*u$ and $\Delta_{\delbar,sw} u=\Delta_\delbar u=\Delta_\del u$. Since $u$ has compact support, $\la\del\del^*u,\del^*\del u\ra=\la\del^*u,\del^*\del^*\del u\ra=0$, therefore
\[
\lv\Delta_\del u\rv^2=\lv\del\del^*u\rv^2+\lv\del^*\del u\rv^2\ge\lv\del\del^*u\rv^2,
\]
and similarly $\lv\Delta_\delbar u\rv\ge\lv\delbar\delbar^*u\rv$. Hence
\[
\lv Qu\rv\le\lv\del\del^*u\rv+\lv\delbar\delbar^*u\rv\le\lv\Delta_\del u\rv+\lv\Delta_\delbar u\rv=2\lv\Delta_{\delbar,sw} u\rv.
\]
Now let $x\in\D(\Delta_{\delbar,sw})$. Since $\Delta_{\delbar,sw}$ is the closure of $(\Delta_\delbar)_0$, there exists a sequence $\{u_k\}_{k\in\N}\subset A^{\bullet,\bullet}_0$ such that $u_k\to x$ and $\Delta_{\delbar,sw} u_k\to\Delta_{\delbar,sw} x$. By the above estimate, $\{Qu_k\}_{k\in\N}$ is a Cauchy sequence, and since $Q$ is closed we get $x\in\D(Q)$ and $Qx=\lim_kQu_k$, which proves \eqref{equation estimate second order part}.

In particular $\D((\Delta_{\delbar,sw})^2)\subseteq\D(\Delta_{\delbar,sw})\subseteq\D(Q)$, so that $\D(\tilde\Delta_{A,ab})=\D((\Delta_{\delbar,sw})^2)$. Moreover, for all $x\in\D((\Delta_{\delbar,sw})^2)$ we have
\[
\lv\Delta_{\delbar,sw} x\rv^2=\la(\Delta_{\delbar,sw})^2x,x\ra\le\lv(\Delta_{\delbar,sw})^2x\rv\lv x\rv\le\left(\frac14\lv(\Delta_{\delbar,sw})^2x\rv+2\lv x\rv\right)^2,
\]
which, together with \eqref{equation estimate second order part}, yields
\[
\lv Qx\rv\le\frac12\lv(\Delta_{\delbar,sw})^2x\rv+4\lv x\rv\qquad\forall x\in\D((\Delta_{\delbar,sw})^2).
\]
Since $\frac12<1$, the Kato-Rellich Theorem \cite[Theorem X.12]{RS2} implies that $\tilde\Delta_{A,ab}=(\Delta_{\delbar,sw})^2+Q$ is self-adjoint on $\D((\Delta_{\delbar,sw})^2)$, and that it is essentially self-adjoint on  $A^{\bullet,\bullet}_0$. Therefore $(\tilde\Delta_A)_0$ is essentially self-adjoint, and its unique self-adjoint extension is $\tilde\Delta_{A,ab}$.

In the Bott-Chern case one argues in the same way with $Q:=(\del+\delbar)^*_{c'}(\del+\delbar)_c$, which acts as $\del^*\del+\delbar^*\delbar$ on $A^{\bullet,\bullet}_0$, using the estimates $\lv\del^*\del u\rv\le\lv\Delta_\del u\rv$ and $\lv\delbar^*\delbar u\rv\le\lv\Delta_\delbar u\rv$ for $u\in A^{\bullet,\bullet}_0$.
\end{proof}

Combining the essential self-adjointness of $(\Delta_{\delbar}^2)_0$ with the K\"ahler identities \eqref{eq kahler identities}, we are able to prove that also the Varouchas Laplacians $(\square_A)_0$ and $(\square_{BC})_0$ are essentially self-adjoint.

\begin{theorem}\label{theorem varouchas essentially self-adjoint}
Let $(M,g)$ be a complete K\"ahler manifold and let $a,b,c\in\{s,w\}$ with $a\le b\le c$. Then $\square_{A,ab}$ and $\square_{BC,bc}$ are the unique self-adjoint extensions of $(\square_A)_0$ and $(\square_{BC})_0$, respectively, and
\[
\D(\square_{A,ab})=\D(\square_{BC,bc})=\D((\Delta_{\delbar,sw})^2).
\]
In particular, they do not depend on the choice of $a,b,c$.
\end{theorem}
\begin{proof}
We prove the statement for $\square_{A,ab}$. Denote by
\[
T_1:=\delbar^*\del^*_{b'}\del\delbar_b,\quad T_2:=\del\delbar_s\delbar^*\del^*_w,\quad T_3:=\delbar\del^*_s\del\delbar^*_w,\quad T_4:=\del\delbar^*_s\delbar\del^*_w
\]
the four addends of $\tilde\Delta_{A,b,4}$, so that, by Subsection \ref{subsection kahler} and \eqref{equation square dolbeault aeppli bc 4}, they satisfy the assumptions of Theorem \ref{theorem sum kohn} and
\begin{gather*}
T_1+T_2+T_3+T_4=\tilde\Delta_{A,b,4}=(\Delta_{\delbar,sw})^2,\\
\D(T_1)\cap\D(T_2)\cap\D(T_3)\cap\D(T_4)=\D((\Delta_{\delbar,sw})^2).
\end{gather*}

We first observe that on a K\"ahler manifold the Varouchas Laplacian $\square_A$ can be written as
\begin{equation}\label{equation varouchas kahler}
\square_A=\delbar^*\del^*\del\delbar+4\,\del\delbar\delbar^*\del^*+\delbar\del^*\del\delbar^*+\del\delbar^*\delbar\del^*.
\end{equation}
Indeed, recall that $\Delta_\delbar=\Delta_\del$ commutes with $\del,\delbar,\del^*,\delbar^*$. Using $\del^2=0$ we get $(\del\del^*)^2=\del(\Delta_\del-\del\del^*)\del^*=\Delta_\delbar\del\del^*$, and similarly $(\delbar\delbar^*)^2=\Delta_\delbar\delbar\delbar^*$. By \eqref{eq kahler identities}, we also have
\[
\del\del^*\delbar\delbar^*=-\del\delbar\del^*\delbar^*=\del\delbar\delbar^*\del^*,\qquad \delbar\delbar^*\del\del^*=-\delbar\del\delbar^*\del^*=\del\delbar\delbar^*\del^*,
\]
so that $(\del\del^*+\delbar\delbar^*)^2=\Delta_\delbar(\del\del^*+\delbar\delbar^*)+2\,\del\delbar\delbar^*\del^*$. Finally,
\[
\delbar\del^*\del\delbar^*=\delbar(\Delta_\del-\del\del^*)\delbar^*=\Delta_\delbar\delbar\delbar^*-\del\delbar\delbar^*\del^*,\qquad
\del\delbar^*\delbar\del^*=\Delta_\delbar\del\del^*-\del\delbar\delbar^*\del^*,
\]
which, recalling $\square_A=\delbar^*\del^*\del\delbar+(\del\del^*+\delbar\delbar^*)^2$, yields \eqref{equation varouchas kahler}.

Now consider the operator $S:=T_1+4T_2+T_3+T_4$ with domain $\D(T_1)\cap\D(T_2)\cap\D(T_3)\cap\D(T_4)=\D((\Delta_{\delbar,sw})^2)$. The operators $T_1,4T_2,T_3,T_4$ satisfy the assumptions of Theorem \ref{theorem sum kohn} as in Subsection \ref{subsection kahler}, so that $S$ is self-adjoint. Moreover, for $x\in\D((\Delta_{\delbar,sw})^2)$ and $i\neq j$ we have $T_ix\in\ker T_j$, hence $\la T_ix,T_jx\ra=\la T_jT_ix,x\ra=0$. Therefore
\[
\lv Sx\rv^2=\lv T_1x\rv^2+16\lv T_2x\rv^2+\lv T_3x\rv^2+\lv T_4x\rv^2,\qquad \lv(\Delta_{\delbar,sw})^2x\rv^2=\sum_{i=1}^4\lv T_ix\rv^2,
\]
which implies
\[
\lv(\Delta_{\delbar,sw})^2x\rv\le\lv Sx\rv\le4\lv(\Delta_{\delbar,sw})^2x\rv\qquad\forall x\in\D((\Delta_{\delbar,sw})^2).
\]
Since $(\Delta_{\delbar,sw})^2$ is the closure of $(\Delta_\delbar^2)_0$, the space $A^{\bullet,\bullet}_0$ is dense in $\D((\Delta_{\delbar,sw})^2)$ with respect to the graph norm of $(\Delta_{\delbar,sw})^2$, hence also with respect to the graph norm of $S$. Moreover, $S$ coincides with $\square_A$ on $A^{\bullet,\bullet}_0$ by \emph{(5)} of Proposition \ref{proposition properties strong weak} and \eqref{equation varouchas kahler}. It follows that $S$ is the closure of $(\square_A)_0$, \textit{i.e.}, $(\square_A)_0$ is essentially self-adjoint. Since $\square_{A,ab}$ is a self-adjoint extension of $(\square_A)_0$, we conclude $\square_{A,ab}=S$ and $\D(\square_{A,ab})=\D((\Delta_{\delbar,sw})^2)$.

In the Bott-Chern case one argues in the same way with the four terms of $\tilde\Delta_{BC,b,4}$, using the identity
\[
\square_{BC}=\del\delbar\delbar^*\del^*+4\,\delbar^*\del^*\del\delbar+\del^*\delbar\delbar^*\del+\delbar^*\del\del^*\delbar,
\]
which follows from similar computations.
\end{proof}

As a corollary, by applying \cite[Theorem 7.6]{HP} we obtain a strong improvement of \cite[Corollary 8.13]{HP}. We present a short proof for the convenience of the reader.

\begin{corollary}\label{corollary del delbar strong weak complete kahler}
Let $(M,g)$ be a complete K\"ahler manifold. Then
\begin{align*}
\del\delbar_s=\del\delbar_w, &&\delbar^*\del^*_s=\delbar^*\del^*_w.
\end{align*}
\end{corollary}
\begin{proof}
Set $P:=(\delbar\oplus\del)_s(\delbar\oplus\del)^*_w$, which is self-adjoint. Since $\im P\subseteq\im(\delbar\oplus\del)_s\subseteq\ker\del\delbar_w$ by \emph{(4)} of Proposition \ref{proposition properties strong weak}, the operators $P$ and $\del\delbar_w$ define a Hilbert complex, whose Laplacian is $P^2+\delbar^*\del^*_s\del\delbar_w=\square_{A,sw}$. By \cite[Lemma 2.11 and p. 98]{BL}, the space $\mathcal E:=\bigcap_{k\ge1}\D(\square_{A,sw}^k)$ is dense in $\D(\del\delbar_w)$ with respect to the graph norm of $\del\delbar_w$. By Theorem \ref{theorem varouchas essentially self-adjoint} and \eqref{equation square dolbeault aeppli bc 4} we have
\[
\mathcal E\subseteq\D(\square_{A,sw})=\D((\Delta_{\delbar,sw})^2)=\D(\tilde\Delta_{A,s,4})\subseteq\D(\delbar^*\del^*_w\del\delbar_s)\subseteq\D(\del\delbar_s).
\]
Let $y\in\D(\del\delbar_w)$ and take $y_k\in\mathcal E$ such that $y_k\to y$ and $\del\delbar_wy_k\to\del\delbar_wy$. Since $\del\delbar_s\subseteq\del\delbar_w$, we have $\del\delbar_sy_k=\del\delbar_wy_k$, and since $\del\delbar_s$ is closed we get $y\in\D(\del\delbar_s)$. Therefore $\del\delbar_s=\del\delbar_w$. The second equality follows by Remark \ref{remark strong-weak}.
\end{proof}

\begin{remark}
To the author’s knowledge, the existing criteria (\textit{e.g.}, completeness, bounded geometry) for proving the equality between minimal and maximal closed extensions of differential operators apply either to first-order operators (see, \textit{e.g.}, \cite[Lemma 4]{AV} or \cite[Section 3.B]{Che}) or to elliptic operators (see, \textit{e.g.}, \cite[Section 3.B]{Che}, or \cite[Corollary 1.1]{Sh}).

Therefore, Corollary \ref{corollary del delbar strong weak complete kahler} constitutes an exception in the literature, since $\del\delbar$ is a second-order operator and is not elliptic. It is clear from the proof that K\"ahlerness plays a key role in this phenomenon.
\end{remark}

\subsection{Reduced \texorpdfstring{$L^2$-$\del\delbar$}{L2-deldelbar}-Lemma on complete K\"ahler manifolds}

As a consequence of \eqref{eq second order laplacian complete kahler} and \eqref{equation square dolbeault aeppli bc 4}, we recall that on a complete K\"ahler manifold all the kernels in $L^2$ of the previously defined Laplacians coincide.

\begin{theorem}[{\cite[Theorem 8.4]{HP}}]\label{theorem complete kahler equality harmonic}
Given a complete K\"ahler manifold $(M,g)$, for all $b\in\{s,w\}$ the following equalities hold on $L^2\Lambda^{p,q}$ and $L^2\Lambda^{k}_\C$
\begin{align*}
\ker d_w\cap\ker d^*_w&=\ker \del_w\cap\ker \del^*_w=\ker \delbar_w\cap\ker \delbar^*_w\\
&=\ker\del\delbar_b\cap\ker\del^*_w\cap\ker\delbar^*_w\\
&=\ker\delbar^*\del^*_b\cap\ker\del_w\cap\ker\delbar_w.
\end{align*}
\end{theorem}

We conclude by providing an alternative presentation and proof of \cite[Corollary 8.6]{HP}. A similar strategy will be used in proving the main result Theorem \ref{theorem l2 del delbar lemma}.

\begin{theorem}\label{theorem reduced l2 del delbar lemma}
Given a complete K\"ahler manifold $(M,g)$, in both the spaces $L^2\Lambda^k_\C$ and $L^2\Lambda^{p,q}$ the following equalities hold true
\begin{align}
\c{\im \del\delbar_s}
&=\c{\im \del_s}\cap\ker\delbar_w=\c{\im \delbar_s}\cap\ker\del_w\label{eq1ddbar}\\
&=\c{\im(\delbar\oplus\del)_s}\cap\ker\del_w\cap\ker\delbar_w\label{eq2ddbar}\\
&=(\c{\im\del_s}+\c{\im\delbar_s})\cap\ker\del_w\cap\ker\delbar_w\label{eq3ddbar}\\
&=\c{\im d_s}\cap\ker\del_w\cap\ker\delbar_w.\label{eq4ddbar}
\end{align}
\end{theorem}
\begin{proof}
Let us intersect all the orthogonal decompositions of the second half of Theorem \ref{theorem orthogonal decompositions} with the space $\ker\del_w\cap\ker\delbar_w$, taking into account the assumption of completeness. Using Proposition \ref{proposition properties strong weak}, Theorem \ref{theorem complete kahler equality harmonic}, Corollary \ref{corollary del delbar strong weak complete kahler} and \eqref{equation kernel}, we obtain
\begin{align}
L^2\Lambda^k_\C\cap\ker\del_w\cap\ker\delbar_w&=\ker \delbar_w\cap\ker \delbar^*_w\overset{\perp}{\oplus}\c{\im d_s}\cap\ker\del_w\cap\ker\delbar_w,\label{eq1decomp}\\
L^2\Lambda^{p,q}\cap\ker\del_w\cap\ker\delbar_w&=\ker \delbar_w\cap\ker \delbar^*_w\overset{\perp}{\oplus}\c{\im \del_s}\cap\ker\delbar_w,\label{eq2decomp}\\
&=\ker \delbar_w\cap\ker \delbar^*_w\overset{\perp}{\oplus}\c{\im \delbar_s}\cap\ker\del_w,\label{eq3decomp}\\
&=\ker \delbar_w\cap\ker \delbar^*_w\overset{\perp}{\oplus}\c{\im(\delbar\oplus\del)_s}\cap\ker\del_w\cap\ker\delbar_w,\label{eq4decomp}\\
&=\ker \delbar_w\cap\ker \delbar^*_w\overset{\perp}{\oplus}\c{\im \del\delbar_s}.\label{eq5decomp}
\end{align}
Decompositions \eqref{eq2decomp}-\eqref{eq5decomp} yield \eqref{eq1ddbar}, \eqref{eq2ddbar} in $L^2\Lambda^{p,q}$. To prove \eqref{eq3ddbar} in $L^2\Lambda^{p,q}$, we argue as follows. By \emph{(2)} and \emph{(4)} of Proposition \ref{proposition properties strong weak}, we have
\[
\c{\im\del\delbar_s}\subseteq (\c{\im\del_s}+\c{\im\delbar_s})\cap\ker\del_w\cap\ker\delbar_w.
\]
On the other hand, 
\[
(\c{\im\del_s}+\c{\im\delbar_s})\cap\ker\del_w\cap\ker\delbar_w\subseteq(\ker\delbar_w\cap\ker\delbar^*_w)^\perp,
\]
so that
\[
(\c{\im\del_s}+\c{\im\delbar_s})\cap\ker\del_w\cap\ker\delbar_w\subseteq \c{\im\del\delbar_s}
\]
by \eqref{eq5decomp}. Adapting the same proof one obtains \eqref{eq4ddbar} in $L^2\Lambda^{p,q}$. Now, \eqref{eq1ddbar}-\eqref{eq3ddbar} hold on $L^2\Lambda^k_\C$ by taking the direct sum for $p+q=k$. Decompositions \eqref{eq2decomp}-\eqref{eq5decomp} hold on $L^2\Lambda^k_\C$ by taking the direct sum for $p+q=k$, therefore  \eqref{eq4ddbar} on $L^2\Lambda^k_\C$ follows from \eqref{eq1decomp}.
\end{proof}

\section{Spectral gap of the elliptic Aeppli and Bott-Chern Laplacians}\label{section spectral gap}

In \cite[Corollary 8.12]{HP} it was proved that on a complete K\"ahler manifold, if the unique self-adjoint extension of the Dolbeault Laplacian $\Delta_{\delbar,sw}$ has a spectral gap, then also the naturally defined self-adjoint extensions of the non-elliptic Aeppli and Bott-Chern Laplacians $\Delta_{A,ab}$, $\Delta_{BC,ab}$, for $a,b\in\{s,w\}$ with $a\le b$, have a spectral gap. In this section, we show that the same holds also for the natural self-adjoint extensions of the elliptic Aeppli and Bott-Chern Laplacians. This holds for both the Kodaira-Spencer and the Varouchas Laplacians. Given the essential self-adjointness of Theorems \ref{theorem kodaira spencer self-adjoint} and \ref{theorem varouchas essentially self-adjoint}, we may restrict to the extensions with indices ${sw}$ for the above operators.

We begin by proving the spectral gap of the Kodaira-Spencer Laplacians.

\begin{theorem}\label{theorem spectral gap kodaira spencer laplacians}
Let $(M,g)$ be a complete K\"ahler manifold. If $\Delta_{\delbar,sw}$ has a spectral gap in $L^2\Lambda^{p,q}$, then the operators $\tilde\Delta_{A,sw}$ and $\tilde\Delta_{BC,sw}$ have a spectral gap in $L^2\Lambda^{p,q}$.
\end{theorem}
\begin{proof}
By Lemma \ref{lemma spectral gap equiv}, since $\Delta_{\delbar,sw}$ has a spectral gap in $L^2\Lambda^{p,q}$, then $(\Delta_{\delbar,sw})^2$ has a spectral gap in $L^2\Lambda^{p,q}$. Thanks to equation \eqref{equation square dolbeault aeppli bc 4}, it follows that $\tilde\Delta_{A,w,4}$ and $\tilde\Delta_{BC,s,4}$ have a spectral gap in $L^2\Lambda^{p,q}$. We now focus on $\tilde\Delta_{BC,s,4}$. 
By Lemma \ref{lemma spectral gap equiv} there is a constant $C>0$ such that
\begin{align*}
\lv\delbar^*\del^*_{w}\psi\rv^2+
\lv\del\delbar_w\psi\rv^2+
\lv\delbar^*\del_w\psi\rv^2+
\lv\del^*\delbar_w\psi\rv^2=
\la \psi,\tilde\Delta_{BC,s,4}\psi\ra
\ge C \la\psi,\psi\ra
\end{align*}
for all $\psi\in\D(\tilde\Delta_{BC,s,4})\cap (\ker\tilde\Delta_{BC,s,4})^\perp\cap L^2\Lambda^{p,q}$. 
Now $\ker\tilde\Delta_{BC,sw}=\ker\Delta_{\delbar,sw}$ by Theorem \ref{theorem complete kahler equality harmonic},  $\ker\Delta_{\delbar,sw}=\ker(\Delta_{\delbar,sw})^2=\ker\tilde\Delta_{BC,s,4}$ by equation \eqref{equation square dolbeault aeppli bc 4} and $\D(\tilde\Delta_{BC,sw})\subseteq\D(\tilde\Delta_{BC,s,4})$ by definition, therefore for all $\psi\in\D(\tilde\Delta_{BC,sw})\cap (\ker\tilde\Delta_{BC,sw})^\perp\cap L^2\Lambda^{p,q}$
\begin{align*}
\la \psi,\tilde\Delta_{BC,sw}\psi\ra&=
\lv\delbar^*\del^*_{w}\psi\rv^2+
\lv\del\delbar_w\psi\rv^2+
\lv\delbar^*\del_w\psi\rv^2+
\lv\del^*\delbar_w\psi\rv^2+
\lv\del_w\psi\rv^2+
\lv\delbar_w\psi\rv^2\\
&\ge
\lv\delbar^*\del^*_{w}\psi\rv^2+
\lv\del\delbar_w\psi\rv^2+
\lv\delbar^*\del_w\psi\rv^2+
\lv\del^*\delbar_w\psi\rv^2
\ge C \la\psi,\psi\ra.
\end{align*}
Thus, again by Lemma \ref{lemma spectral gap equiv}, $\tilde\Delta_{BC,sw}$ has a spectral gap in $L^2\Lambda^{p,q}$. A similar argument works for $\tilde\Delta_{A,sw}$.
\end{proof}

We now use Theorem \ref{theorem spectral gap kodaira spencer laplacians} to obtain an alternative proof of \cite[Theorem 8.10]{HP}.

\begin{theorem}\label{theorem closed image del delbar}
Let $(M,g)$ be a complete K\"ahler manifold. If $\Delta_{\delbar,sw}$ has a spectral gap in $L^2\Lambda^{p,q}$, then for $b\in\{s,w\}$ it follows that
$\im\del\delbar_b$ and $\im\delbar^*\del^*_b$ are closed in $L^2\Lambda^{p,q}$.
\end{theorem}
\begin{proof}
By Theorem \ref{theorem spectral gap kodaira spencer laplacians}, $\tilde\Delta_{BC,bw}$ has a spectral gap in $L^2\Lambda^{p,q}$. By Lemma \ref{lemma spectral gap equiv}, this is equivalent to the closedness of the image of $\tilde\Delta_{BC,bw}$. Therefore we can orthogonally decompose
\begin{align*}
L^2\Lambda^{p,q}&=\ker\tilde\Delta_{BC,bw}\overset{\perp}{\oplus}\im\tilde\Delta_{BC,bw}\\
&\subseteq \ker\tilde\Delta_{BC,bw}\overset{\perp}{\oplus}\left(\im\del\delbar_b + \im \delbar^*\del^*_s+\im \del^*\delbar_s+\im\delbar^*\del_s+\im(\del^*\oplus\delbar^*)_s\right)\\
&\subseteq \ker\tilde\Delta_{BC,bw}\overset{\perp}{\oplus}\im\del\delbar_b \overset{\perp}{\oplus} \c{\im(\del^*\oplus\delbar^*)_s}\subseteq L^2\Lambda^{p,q}.
\end{align*}
The first inclusion follows by definition of $\tilde\Delta_{BC,bw}$, while the second inclusion is justified by \emph{(2)} of Proposition \ref{proposition properties strong weak} and the fact that 
\[
\c{\im\del^*_s}+\c{\im\delbar^*_s}\subseteq\c{\im(\del^*\oplus\delbar^*)_s}.
\]
The orthogonality of the above decompositions is easy to check.
The closedness of $\im\del\delbar_b$ now follows from the orthogonality of the decomposition of $L^2\Lambda^{p,q}$ in the last line. Using the spectral gap of $\tilde\Delta_{A,sb}$ we similarly find the closedness of $\im\delbar^*\del^*_b$ in $L^2\Lambda^{p,q}$.
\end{proof}

We are now ready to prove the spectral gap also of Varouchas Laplacians.

   \begin{theorem}\label{theorem spectral gap varouchas laplacians}
   Let $(M,g)$ be a complete K\"ahler manifold. If $\Delta_{\delbar,sw}$ has a spectral gap in $L^2\Lambda^{p,q}$, then $\square_{A,sw}$ and $\square_{BC,sw}$ have a spectral gap in $L^2\Lambda^{p,q}$.
   \end{theorem}
   \begin{proof}
   We prove the statement for $\square_{A,sw}$, the Bott-Chern case being analogous. By Lemma \ref{lemma spectral gap equiv}, $(\Delta_{\delbar,sw})^2$ has a spectral gap in $L^2\Lambda^{p,q}$, \textit{i.e.}, there is $C>0$ such that $\la(\Delta_{\delbar,sw})^2x,x\ra\ge C\la x,x\ra$ for all $x\in\D((\Delta_{\delbar,sw})^2)\cap(\ker(\Delta_{\delbar,sw})^2)^\perp\cap L^2\Lambda^{p,q}$. With the notation of the proof of Theorem \ref{theorem varouchas essentially self-adjoint}, we have $\square_{A,sw}=T_1+4T_2+T_3+T_4$ and $(\Delta_{\delbar,sw})^2=T_1+T_2+T_3+T_4$ on the common domain $\D((\Delta_{\delbar,sw})^2)$, with $T_1,\dots,T_4$ non-negative operators. Hence $\ker\square_{A,sw}=\ker T_1\cap\dots\cap\ker T_4=\ker(\Delta_{\delbar,sw})^2$ and
   \[
   \la\square_{A,sw}x,x\ra\ge\la(\Delta_{\delbar,sw})^2x,x\ra\ge C\la x,x\ra
   \]
   for all $x\in\D(\square_{A,sw})\cap(\ker\square_{A,sw})^\perp\cap L^2\Lambda^{p,q}$. The claim follows by Lemma \ref{lemma spectral gap equiv}.
   \end{proof}

\section{\texorpdfstring{$L^2$-$\del\delbar$}{L2-deldelbar}-Lemma}\label{section l2 del delbar lemma}

A consequence of the spectral gap of the self-adjoint extensions of the elliptic Aeppli and Bott-Chern Laplacians is a $\del\delbar$-Lemma holding on smooth $L^2$-forms, which improves Theorem \ref{theorem reduced l2 del delbar lemma}.

\begin{theorem}\label{theorem l2 del delbar lemma}
    Let $(M,g)$ be a complete K\"ahler manifold such that $\Delta_{\delbar,sw}$ has a spectral gap in $L^2\Lambda^{k}_\C$. Given a smooth $L^2$-form $\alpha\in L^2A^k_\C$ which satisfies $\del\alpha=\delbar\alpha=0$, the following conditions are equivalent:
    \begin{enumerate}
        \item $\alpha=\del\delbar\beta$, for $\beta\in L^2 A^{k-2}_\C$;
        \item $\alpha=\del\gamma$, for $\gamma\in L^2 A^{k-1}_\C$;
        \item $\alpha=\delbar\zeta$, for $\zeta\in L^2 A^{k-1}_\C$;
        \item $\alpha=\del\eta+\delbar\theta$, for $\eta,\theta\in L^2 A^{k-1}_\C$;
        \item $\alpha=d\lambda$, for $\lambda\in L^2 A^{k-1}_\C$.
    \end{enumerate}
\end{theorem}
\begin{proof}
    We define the orthogonal projection
    \[
    P_{BC}:L^2\Lambda^{k}_\C\to\ker \square_{BC,sw}
    \]
    and the Green operator
    \begin{align*}
    G_{BC}:&\im \square_{BC,sw}\to \D(\square_{BC,sw})\cap\ker(\square_{BC,sw})^\perp\\
    & \square_{BC,sw}\beta\mapsto\beta-P_{BC}\beta.
    \end{align*}
    Since $\Delta_{\delbar,sw}=\Delta_{\del,sw}=\frac12\Delta_{d,sw}$ has a spectral gap in $L^2\Lambda^{k}_\C$, by Theorem \ref{theorem spectral gap varouchas laplacians} also $\square_{BC,sw}$ and $\square_{A,sw}$ have a spectral gap in $L^2\Lambda^{k}_\C$.
    Being $\im \square_{BC,sw}$ closed, recalling the orthogonal decomposition
    \[
    L^2\Lambda^{k}_\C=\ker \square_{BC,sw}\oplus\im \square_{BC,sw},
    \]
    we can extend the Green operator on the whole space $L^2\Lambda^{k}_\C$ simply by setting $G_{BC}=0$ on $\ker \square_{BC,sw}$. Given any form $\alpha\in L^2\Lambda^{k}_\C$, it decomposes as
    \[
    \alpha=P_{BC}\alpha+(id-P_{BC})\alpha,
    \]
    where $(id-P_{BC})\alpha=\square_{BC,sw}\beta$, and we can assume $\beta\in \ker(\square_{BC,sw})^\perp$ by replacing it with $\beta-P_{BC}\beta$. Now $G_{BC}\alpha=\beta$, therefore
    \[
    \alpha=P_{BC}\alpha+\square_{BC,sw}G_{BC}\alpha.
    \]
    If, in particular, $\alpha$ is smooth and so $\alpha\in L^2A^k_\C$, then by elliptic regularity (see, \textit{e.g.}, \cite[Theorem 2.1]{HP}) also $G_{BC}\alpha=\beta\in L^2A^k_\C\cap \D(\square_{BC,sw})$ is smooth. In particular, by \emph{(5)} of Proposition \ref{proposition properties strong weak} we obtain 
    \[
    \square_{BC,sw}G_{BC}\alpha=\square_{BC}G_{BC}\alpha.
    \]
    The same holds for analogous operators $P_\delta$ and $G_\delta$ for $\delta\in\{A,\del,\delbar,d\}$.
    Therefore, if $\alpha\in L^2A^k_\C$, for $\delta\in\{A,BC\}$ and $\epsilon\in\{\del,\delbar,d\}$ we deduce decompositions
    \[
    \alpha=P_{\delta}\alpha+\square_{\delta}G_{\delta}\alpha=P_\epsilon\alpha+\Delta_{\epsilon}G_{\epsilon}\alpha,
    \]
    with smooth forms $G_{\delta}\alpha,G_{\epsilon}\alpha\in L^2A^k_\C$. Recall that by Theorem \ref{theorem complete kahler equality harmonic} the kernels of $\Delta_{\delbar,sw}=\Delta_{\del,sw}=\frac12\Delta_{d,sw}$, $\square_{BC,sw}$ and $\square_{A,sw}$ coincide, thus $P_{BC}\alpha=0$ iff $P_\mu\alpha=0$ for all $\mu\in\{A,\del,\delbar,d\}$.

    Now assume that $\alpha\in L^2A^k_\C$ satisfies $\del\alpha=\delbar\alpha=0$. We know $\ker\del_w\cap\ker\delbar_w\subseteq L^2\Lambda^k_\C$ is $L^2$-orthogonal to
    \begin{align*}
    &\im d_s^*, &&\im\del_s^*, & &\im\delbar_s^*, &\im(\del^*\oplus\delbar^*)_s, 
    &&\im\delbar^*\del_s^*.
    \end{align*}
    Thanks to this, Lemma \ref{lemma decomposition hilbert complexes} and to \emph{(5)} of Proposition \ref{proposition properties strong weak}, we obtain
    \begin{align*}
    &\alpha=\square_{BC}G_{BC}\alpha &&\iff && \alpha=\del\delbar\delbar^*\del^*G_{BC}\alpha,\\
    &\alpha=\square_{A}G_{A}\alpha &&\iff && \alpha=(\delbar\delbar^*+\del\del^*)^2G_{A}\alpha,\\
    &\alpha=\Delta_{\del}G_{\del}\alpha &&\iff && \alpha=\del\del^*G_{\del}\alpha,\\
    &\alpha=\Delta_{\delbar}G_{\delbar}\alpha &&\iff && \alpha=\delbar\delbar^*G_{\delbar}\alpha,\\
    &\alpha=\Delta_{d}G_{d}\alpha &&\iff && \alpha=dd^*G_{d}\alpha.
    \end{align*}
    It is then enough to set $\beta:=\delbar^*\del^*G_{BC}\alpha$, $\gamma:=\del^*G_{\del}\alpha$, $\zeta:=\delbar^*G_{\delbar}\alpha$, $\eta:=\del^*(\delbar\delbar^*+\del\del^*) G_{A}\alpha$, $\theta:=\delbar^*(\delbar\delbar^*+\del\del^*) G_{A}\alpha$, $\lambda:=d^*G_{d}\alpha$, which lie in $L^2\Lambda^\bullet_\C$ since the image of each Green operator is the domain of the corresponding Laplacian.
    This concludes the proof.
\end{proof}

\begin{remark}
    If, on a complete K\"ahler manifold $(M,g)$, $\Delta_{\delbar,sw}$ has a spectral gap only in $L^2\Lambda^{p,q}$, then conditions (1)-(5) 
    are equivalent provided that $\alpha\in L^2 A^{p,q}$ and $k=p+q$. The same holds for any direct sum of spaces $L^2\Lambda^{p,q}$ with $k=p+q$.
\end{remark}

\begin{remark}\label{remark bounded geometry}
    Suppose that, in addition to the assumptions of Theorem \ref{theorem l2 del delbar lemma}, the manifold $(M,g)$ is of bounded geometry and the elliptic Laplacians $\Delta_{\delbar}=\Delta_{\del}=\frac12\Delta_{d}$, $\square_{BC}$ and $\square_{A}$ are $\cinf$-bounded and uniformly elliptic, with the bundles of differential forms which are of bounded geometry (cf. \cite{Sh}). Examples of such manifolds are given by the universal coverings of K\"ahler hyperbolic manifolds \cite{G}.  Then, by \cite[Proposition 1.1]{Sh} the domains of the above operators coincide with the Sobolev spaces
    \[
    \D(\Delta_{\delbar,sw})=\D(\Delta_{\del,sw})=\D(\Delta_{d,sw})=W^2_2(M,\Lambda^\bullet_\C)
    \]
    and
    \[
    \D(\square_{A,sw})=\D(\square_{BC,sw})=W^4_2(M,\Lambda^\bullet_\C),
    \]
    with the notation for Sobolev spaces as in \cite{Sh}. Therefore, since the image of each Green operator is one of the above domains, and thus the image is a Sobolev space, in the statement of Theorem \ref{theorem l2 del delbar lemma} we can choose $\beta\in W^2_2(M,\Lambda^{k-2}_\C)$ and $\gamma,\zeta,\eta,\theta,\lambda\in W^1_2(M,\Lambda^{k-1}_\C)$. In particular, $\del\beta,\delbar\beta\in W^1_2(M,\Lambda^{k-1}_\C)\subseteq L^2\Lambda_\C^{k-1}$.
\end{remark}

\end{document}